\documentclass{article}
\usepackage{amsmath}
\usepackage{amsfonts}
\usepackage{amssymb}

\newtheorem{theorem}{Theorem}

\newtheorem{lemma}[theorem]{Lemma}

\newenvironment{proof}[1][Proof]{\noindent\textbf{#1.} }{\ \rule{0.5em}{0.5em}}

\begin{document}

\title{Thirty-six Officers and their Code}
\author{Harold N. Ward\\
Department of Mathematics\\
University of Virginia\\
Charlottesville, VA 22904\\
USA}
\maketitle

\begin{abstract}
This note presents a short proof of Euler's 36 officer conjecture. This
implies that there is no affine plane of order $6$, but we also give a
direct proof.
\end{abstract}

\section*{Introduction}

Leonhard Euler published his famous 36 officers problem in 1782 (a
translation of the paper is listed under \cite{E}): \textquotedblleft This
question concerns a group of thirty-six officers of six different ranks,
taken from six different regiments, and arranged in a square in a way such
that in each row and column there are six officers, each of a different rank
and regiment.\textquotedblright\ Euler thought there was no such arrangement
and conjectured that analogous ones for rank-regiment counts that leave a
remainder of 2 when divided by 4 were all impossible. His conjecture for the
original six was proved correct by G. Tarry in 1900 \cite{T}. But in 1959,
E. T. Parker \cite{P} showed that the problem is solvable for an infinite
subset of those counts, including 10. Then in 1960, Parker, R. C. Bose, and
S. S. Shrikhande \cite{BPS} proved the problem solvable for \emph{all }%
rank-regiment counts other than two and six.

There is a combinatorial proof of the 36 officer impossibility by D. R.
Stinson \cite{S} and a coding-theory one by S. T. Dougherty \cite{D}
(incorporating part of Stinson's proof). The present note follows the
general outline of these two papers, but it contains some different ways of
setting up details. They will be itemized in steps.

It is well-documented \cite[Section III.3]{HCD} that the officer problem is
equivalent to one involving \textbf{nets} (among other structures). An $%
(n,k) $ net is a combinatorial design $(\mathcal{P},\mathcal{L})$ consisting
of a set $\mathcal{P}$ of $n^{2}$ \textbf{points} and a collection\emph{\ }$%
\mathcal{L}$ of \textbf{lines} that are $n$-subsets of $\mathcal{P}$. The
lines have these properties:

\begin{enumerate}
\item $\mathcal{L}$ is the disjoint union of $k$ \textbf{parallel classes}.
Each class is a partition of $\mathcal{P}$ into $n$ lines.

\item Two lines from different parallel classes meet at exactly one point.
\end{enumerate}

\noindent A solution to the 36 officer problem is equivalent to the
existence of a $(6,4)$ net, the positions in the square being the points.
The rows and columns provide two parallel classes, the locations of officers
by regiment form the lines of the third class, and the locations by rank
form the fourth. Here is an illustration of a square given by Euler in \cite%
{E} that almost works. Latin letters denote the regiments and Greek the
ranks, Euler's traditional symbols leading to the name \textquotedblleft
Graeco-Latin square.\textquotedblright\ Unfortunately, the pairs $b\zeta $
and $d\varepsilon $ are duplicated (and $b\varepsilon $ and $d\zeta $ left
out).%
\[
\begin{tabular}{|c|c|c|c|c|c|}
\hline
$a\alpha $ & $b\zeta $ & $c\delta $ & $d\varepsilon $ & $e\gamma $ & $f\beta 
$ \\ \hline
$b\beta $ & $c\alpha $ & $f\varepsilon $ & $e\delta $ & $a\zeta $ & $d\gamma 
$ \\ \hline
$c\gamma $ & $d\varepsilon $ & $a\beta $ & $b\zeta $ & $f\delta $ & $e\alpha 
$ \\ \hline
$d\delta $ & $f\gamma $ & $e\zeta $ & $c\beta $ & $b\alpha $ & $a\varepsilon 
$ \\ \hline
$e\varepsilon $ & $a\delta $ & $b\gamma $ & $f\alpha $ & $d\beta $ & $c\zeta 
$ \\ \hline
$f\zeta $ & $e\beta $ & $d\alpha $ & $a\gamma $ & $c\varepsilon $ & $b\delta 
$ \\ \hline
\end{tabular}%
\]

\section*{Step One}

The \textbf{code} of a net $(\mathcal{P},\mathcal{L})$ over a field $\mathbb{%
F}$ is the subspace of the $\mathbb{F}$-space $\mathbb{F}^{\mathcal{P}}$ of $%
\mathbb{F}$-valued functions on $\mathcal{P}$ spanned by the characteristic
functions of the lines \cite{D}. Following Assmus and Key \cite[Definition
1.2.5]{AK}, we denote the characteristic function of a subset $X$ of $%
\mathcal{P}$ by $v^{X}$. If $\Pi $ is a parallel class, $v^{\mathcal{P}%
}=\sum_{\lambda \in \Pi }v^{\lambda }$, and $v^{\lambda }v^{\lambda ^{\prime
}}=0$ for $\lambda $ and $\lambda ^{\prime }$ different members of $\Pi $.
The \textbf{weight} $\mathrm{wt}(f)$ of $f\in \mathbb{F}^{\mathcal{P}}$ is
the number of points $P$ with $f(P)$ nonzero; and the standard dot product
on $\mathbb{F}^{\mathcal{P}}$ is given by $f\cdot f^{\prime }=\sum_{P\in 
\mathcal{P}}f(P)f^{\prime }(P)$. Let $\mathcal{N}$ be a $(6,4)$ net $(%
\mathcal{P},\mathcal{L})$ and let $\mathcal{C}$ be its binary code ($\mathbb{%
F=F}_{2}$). As in \cite{D} and \cite{S}, we eventually show that $\mathcal{N}
$ cannot exist by establishing contradictory information about the dimension
of $\mathcal{C}$.

The \textbf{hull} $\mathcal{H}$ of $\mathcal{N}$ at $\mathbb{F}_{2}$ is $%
\mathcal{C}\cap \mathcal{C^{\perp }}$ (the orthogonal space $\mathcal{%
C^{\perp }}$ taken with respect to the dot product) \cite[Definition 2.4.3]%
{AK}. It contains the differences (sums in this binary case) of parallel
lines. If $\Pi _{1},\ldots ,\Pi _{4}$ are the parallel classes and $\lambda
_{i}\in \Pi _{i}$, then the span $\left\langle \lambda _{1},\ldots ,\lambda
_{4}\right\rangle $ has dimension 4, because the Gram matrix $\left[ \lambda
_{i}\cdot \lambda _{j}\right] $ is 
\[
\left[ 
\begin{array}{cccc}
0 & 1 & 1 & 1 \\ 
1 & 0 & 1 & 1 \\ 
1 & 1 & 0 & 1 \\ 
1 & 1 & 1 & 0%
\end{array}%
\right] 
\]
which is nonsingular. Thus $\mathcal{C}=\left\langle \lambda _{1},\ldots
,\lambda _{4}\right\rangle \perp \mathcal{H}$, since for $\lambda \in \Pi
_{i}$, $\lambda -\lambda _{i}\in \mathcal{H}$.

\begin{lemma}
We have $\dim \mathcal{C}\leq 20$.
\end{lemma}

\begin{proof}
Since $\mathcal{C}\subseteq \mathcal{H}^{\perp }$, $\dim \mathcal{C}\leq
36-\dim \mathcal{H}=36-(\dim \mathcal{C}-4)$, making $\dim \mathcal{C}\leq 20$.
\end{proof}

Because $v^{\mathcal{P}}=\sum_{\lambda \in \Pi_i}v^{\lambda }$, there are
parallel class dependencies%
\begin{equation}
\sum_{\lambda \in \Pi_i}v^{\lambda }-\sum_{\lambda \in \Pi_j}v^{\lambda
}=0,\quad i\neq j,  \label{EqClassDep}
\end{equation}%
spanned by the three with $j=4$. The next step in the proof of Euler's
conjecture is to show that there are no further independent dependencies.
That will imply that $\dim \mathcal{C}\geq 24-3=21$ and contradict the Lemma.

\section*{Step Two}

To that end, consider sets of lines, and ascribe to such a set $\Lambda $
its \textbf{parallax} $\pi (\Lambda )=l_{1}l_{2}l_{3}l_{4}$, with $%
l_{i}=\left\vert \Lambda \cap \Pi _{i}\right\vert $. Call the points on the
lines of $\Lambda $ the \textbf{points} of $\Lambda $, and let $p_{j}$ be
the number of $j$-points, those that appear on $j$ lines of $\Lambda $. Put $%
l=l_{1}+\ldots +l_{4}$ and $m=\sum_{i<j}l_{i}l_{j}$. Then  double-counting
incident point-line pairs gives%
\begin{eqnarray*}
6l &=&p_{1}+2p_{2}+3p_{3}+4p_{4} \\
m &=&p_{2}+3p_{3}+6p_{4}.
\end{eqnarray*}%
Consequently%
\begin{eqnarray*}
p_{1} &=&6l-2m+3p_{3}+8p_{4} \\
p_{2} &=&m-3p_{3}-6p_{4}.
\end{eqnarray*}%
Let $c(\Lambda )$ be the binary sum $\sum_{\lambda \in \Lambda }v^{\lambda }$%
, the member of $\mathcal{C}$ spanned by the lines of $\Lambda $. It is a
consequence of the class dependencies (\ref{EqClassDep}) that if we change $%
\Lambda $ to $\Lambda ^{\prime }$ by switching the lines of $\Lambda $ in $%
\Pi _{i}$ with those in $\Pi _{i}$ not in $\Lambda $, for an even number of
parallel classes $\Pi _{i}$, then $c(\Lambda ^{\prime })=c(\Lambda )$. The
corresponding line counts $l_{i}$ and $l_{i}^{\prime }$ of $\Lambda $ and $%
\Lambda ^{\prime }$ satisfy $l_{i}+l_{i}^{\prime }=6$. A line dependency
corresponds to a line set $\Lambda $ with $c(\Lambda )=0$. For such a set,
what are the possibilities for $\pi (\Lambda )$? For instance, the parallel
class dependencies have $\pi (\Lambda )=6600,6060,\ldots ,6666$, with an
even number of 0s and 6s.

In searching for the parallaxes corresponding to line sets $\Lambda $ with $%
c(\Lambda )=0$, we may assume that in $\pi (\Lambda )$, $3\geq l_{1}\geq
l_{2}\geq l_{3}$, by switchings involving $\Pi _{4}$ and then renumbering.
Moreover, if $l_{1}=3$, we can also take $3=l_{1}\geq l_{2}\geq l_{3}\geq
l_{4}$. The 1's of any $c(\Lambda )$ appear at the 1-points and the 3-points
of $\Lambda $, and%
\begin{equation}
\mathrm{wt}(c(\Lambda ))=p_{1}+p_{3}=6l-2m+4p_{3}+8p_{4}.  \label{EqWt}
\end{equation}%
If $c(\Lambda )$ is to be 0, we need $p_{1}=p_{3}=0$. That gives%
\begin{eqnarray*}
2p_{2} &=&9l-m \\
4p_{4} &=&m-3l.
\end{eqnarray*}%
Demanding that $p_{2}$ and $p_{4}$ be nonnegative integers for parallaxes
satisfying the inequalities listed gives four possibilities, as a short
Maple computation shows:%
\[
2222,\quad 2226,\quad 3330,\quad 3332.
\]

All but $2222$ can be ruled out. Switching will change $2226$ to $2240$. But
if $\pi (M)=2240$, then $\mathrm{wt}(c(M))=4p_{3}+8p_{4}+8$, by (\ref{EqWt}%
), which cannot be 0. If $\pi (\Lambda )=3330$ or $3332$ and $c(\Lambda )=0$%
, then adding or removing a line in $\Pi _{4}$ gives a line set $M$ with $%
\pi (M)=3331$ and $\mathrm{wt}(c(M))=6$. But $\mathrm{wt}%
(c(M))=4p_{3}+8p_{4}-12$, a multiple of 4; so $3330$ and $3332$ are out, too.

When $c(\Lambda )=0$ and $\pi (\Lambda )=2222$, $p_{1}=p_{3}=0$, $p_{2}=24$,
and $p_{4}=0$. This line set comes up in \cite{D}, where it is ruled out. We
shall present another argument to exclude it in the next step. Before that,
begin by labeling the two lines in $\Lambda \cap \Pi _{i}$ with 1 and $-1$.
Then the following facts are all consequences of the intersection properties
of the $(6,4)$ net. Each of the 24 points of $\Lambda $ can be tagged by the
two lines of $\Lambda $ through it, in a quadruple $x_{1}x_{2}x_{3}x_{4}$.
Two of the $x_{i}$ are $0$ and the other two $\pm 1$. For example, $10\text{-%
}10$ means the point on line $1$ of $\Lambda \cap \Pi _{1}$ and line $-1$ of 
$\Lambda \cap \Pi _{3}$. (If these quadruples are interpreted as points in $%
\mathbb{R}^4$, they are the vertices of a regular $24$-cell \cite[Section 3.7%
]{C}.)  Each line not in $\Lambda $ goes through three of the 24 points (the
six intersections with the lines of $\Lambda $ not parallel to it, doubled
up), and it can be labeled by them. If, say, the line is in $\Pi _{4}$, its
labeling points will be $xy00,\text{-}x0z0,0\text{-}y\text{-}z0$, where $%
x,y,z\in \left\{ 1,-1\right\} $. This and the other three lines of $\Pi _{4}$
not in $\Lambda $ will then be%
\begin{eqnarray*}
&&xy00\quad \text{-}x0z0\quad 0\text{-}y\text{-}z0 \\
&&x\text{-}y00\quad \text{-}x0\text{-}z0\quad 0yz0 \\
&&\text{-}xy00\quad x0\text{-}z0\quad 0\text{-}yz0 \\
&&\text{-}x\text{-}y00\quad x0z0\quad 0y\text{-}z0.
\end{eqnarray*}%
It follows that there are just two possible line lists, one containing $1100$%
, $\text{-}1010$, \newline
$0\text{-}1 \text{-}10$, and the other, $1100$, $\text{-}10\text{-}10$, $0%
\text{-}110$. Moreover, the second is obtained from the first by negating
all entries. On the other hand, upon exchanging the signs for $\Lambda \cap
\Pi _{4}$, the displayed list will not change, but the new line lists for
each $\Pi _{i}-(\Lambda \cap \Pi _{i})$, $i<4$, will be obtained by negating
all entries. What this implies is that all the possible collections of the
four parallel class line lists for the lines not in $\Lambda $ are
equivalent under sign changes, that is, label changes of the members of the $%
\Lambda \cap \Pi _{i}$.

\section*{Step Three}

Set up the standard layout for a six by six Graeco-Latin square. As before,
the 36 small squares represent the points of the net, and rows and columns
correspond to the first two parallel classes. The lines of the third are $%
a,b,c,d,e,f$, and those of the fourth, $\alpha ,\beta ,\gamma ,\delta
,\varepsilon ,\zeta $. The lines in $\Lambda $ are the right two columns, $%
x_{1}=\pm 1$; the bottom two rows, $x_{2}=\pm 1$; $e$ and $f$, $x_{3}=\pm 1$%
; and $\varepsilon $ and $\zeta $, $x_{4}=\pm 1$. Because of the labeling
flexibility for the other lines, outlined above, we can include the 24
points of $\Lambda $ in the diagram:%
\[
\begin{tabular}{|c|c|c|c|c|c|}
\hline
$0011$ &  &  &  & $100\text{-}1$ & $\text{-}10\text{-}10$ \\ \hline
& $00\text{-}11$ &  &  & $1010$ & $\text{-}100\text{-}1$ \\ \hline
&  & $001\text{-}1$ &  & $10\text{-}10$ & $\text{-}1001$ \\ \hline
&  &  & $00\text{-}1\text{-}1$ & $1001$ & $\text{-}1010$ \\ \hline
$01\text{-}10$ & $010\text{-}1$ & $0101$ & $0110$ & $1100$ & $\text{-}1100$
\\ \hline
$0\text{-}10\text{-}1$ & $0\text{-}110$ & $0\text{-}1\text{-}10$ & $0\text{-}%
101$ & $1\text{-}100$ & $\text{-}1\text{-}100$ \\ \hline
\end{tabular}%
\]
Continuing with the flexibility, we assign the points of $\Lambda$ on the
remaining lines:%
\[
\begin{tabular}{ll}
$a:1100,0\text{-}10\text{-}1,\text{-}1001$ & $\alpha :1100,0\text{-}110,%
\text{-}10\text{-}10$ \\ 
$b:\text{-}1100,0\text{-}101,100\text{-}1$ & $\beta :\text{-}1100,0\text{-}1%
\text{-}10,1010$ \\ 
$c:1\text{-}100,0101,\text{-}100\text{-}1$ & $\gamma :1\text{-}100,01\text{-}%
10,\text{-}1010$ \\ 
$d:\text{-}1\text{-}100,010\text{-}1,1001$ & $\delta :\text{-}1\text{-}%
100,0110,10\text{-}10$%
\end{tabular}%
\]%
The result, on filling in the line names, is two-thirds of a Graeco-Latin
square:%
\[
\begin{tabular}{|c|c|c|c|c|c|}
\hline
$e\varepsilon $ &  &  &  & $b\zeta $ & $f\alpha $ \\ \hline
& $f\varepsilon $ &  &  & $e\beta $ & $c\zeta $ \\ \hline
&  & $e\zeta $ &  & $f\delta $ & $a\varepsilon $ \\ \hline
&  &  & $f\zeta $ & $d\varepsilon $ & $e\gamma $ \\ \hline
$f\gamma $ & $d\zeta $ & $c\varepsilon $ & $e\delta $ & $a\alpha $ & $b\beta 
$ \\ \hline
$a\zeta $ & $e\alpha $ & $f\beta $ & $b\varepsilon $ & $c\gamma $ & $d\delta 
$ \\ \hline
\end{tabular}%
\]

Now the challenge is to fill in the twelve blank squares with pairs from $%
\left\{ a,b,c,d\right\} \times \left\{ \alpha ,\beta ,\gamma ,\delta
\right\} $ with the desired non-repetition properties. So $a\alpha ,b\beta
,c\gamma ,d\delta $ are excluded, being present in the lower right, and each
row and column has further exclusions coming from the bottom two rows and
the right two columns. We abbreviate the layout this way, showing the row
and column exclusions at the right and bottom sides:%
\[
\begin{tabular}{ccccc}
\cline{1-4}
\multicolumn{1}{|c}{$\times $} & \multicolumn{1}{|c}{} & \multicolumn{1}{|c}{
} & \multicolumn{1}{|c}{} & \multicolumn{1}{|c}{$b,\alpha $} \\ 
\cline{1-4}\cline{1-2}\cline{3-4}
\multicolumn{1}{|c}{} & \multicolumn{1}{|c}{$\times $} & \multicolumn{1}{|c}{
} & \multicolumn{1}{|c}{} & \multicolumn{1}{|c}{$c,\beta $} \\ \cline{1-4}
\multicolumn{1}{|c}{} & \multicolumn{1}{|c}{} & \multicolumn{1}{|c}{$\times $%
} & \multicolumn{1}{|c}{} & \multicolumn{1}{|c}{$a,\delta $} \\ \cline{1-4}
\multicolumn{1}{|c}{} & \multicolumn{1}{|c}{} & \multicolumn{1}{|c}{} & 
\multicolumn{1}{|c}{$\times $} & \multicolumn{1}{|c}{$d,\gamma $} \\ 
\cline{1-4}
$a,\gamma $ & $d,\alpha $ & $c,\beta $ & $b,\delta $ & 
\end{tabular}%
\]%
For instance, none of the three pairs in the top row can involve $b$ or $%
\alpha $; and none in the left column $a$ or $\gamma $. (The four squares
with $\times $'s are to be left blank.)

The second row and the third column both have the same exclusions, $c$ and $%
\beta $. So the pairs available for the \textquotedblleft
cross\textquotedblright\ of the five squares of that row and that column are 
$a\gamma ,a\delta ,b\alpha ,b\gamma ,b\delta ,d\alpha ,$ and $d\gamma $. The
pair $b\gamma $ can go only in the center square of the cross, because of
the exclusions governing its other squares. But that means the four pairs $%
a\gamma ,b\alpha ,b\delta ,$ and $d\gamma $ cannot appear in the cross, and
that leaves only two pairs for the other four positions. So $b\gamma $
cannot appear in the cross. Now whatever pair is in the center rules out two
other pairs, leaving only three pairs for the four remaining squares of the
cross. Thus the challenge cannot be met.

In conclusion, the assumption of a further dependency on the lines of a $%
(6,4)$ net beyond the parallel class dependencies, which necessarily
involves a line set with parallax $2222$, has been shown to be untenable.
Thus there is no $(6,4)$ net and Euler's 36 officer conjecture is indeed
correct!

\section*{No affine plane of order $6$}

Since four parallel classes of an affine plane of order $6$ would constitute
a $(6,4)$ net, there can be no such plane. A. Bichara \cite{B} gave a direct
proof from the incidence properties of such a plane,, relating them to arcs.
Here we give a different direct proof.

Let $\mathbf{A}$ be a hypothetical affine plane of order $6$. We first show
that the diagonals of any parallelogram in $\mathbf{A}$ are parallel.
Suppose not, and set up a coordinate system for $\mathbf{A}$ using $1,\ldots
,6$ for the coordinates and arranging things so that the sides of the
offending parallelogram are the lines $x=1,x=3,y=1,y=3$, the line $R:y=x$ is
one of the diagonals, and it and the other diagonal $D$ meet in $11$ (we
shall abbreviate $(x,y)$ to $xy$). By permuting $4,5,6$, we can arrange $R$
and $D$ to contain these points:%
\begin{eqnarray*}
R &:&11,22,33,44,55,66 \\
D &:&13,22,31,46,54,65.
\end{eqnarray*}%
Here's a schematic diagram of the coordinate grid, with the points of $R$
and $D$ indicated:%
\[
\begin{tabular}{lllllll}
\cline{2-7}\cline{4-4}
$y=6$ & \multicolumn{1}{|l}{} & \multicolumn{1}{|l}{} & \multicolumn{1}{|l}{}
& \multicolumn{1}{|l}{$D$} & \multicolumn{1}{|l}{} & \multicolumn{1}{|l|}{$R$%
} \\ \cline{2-7}\cline{4-4}
$y=5$ & \multicolumn{1}{|l}{} & \multicolumn{1}{|l}{} & \multicolumn{1}{|l}{}
& \multicolumn{1}{|l}{} & \multicolumn{1}{|l}{$R$} & \multicolumn{1}{|l|}{$D$%
} \\ \cline{2-7}\cline{4-4}
$y=4$ & \multicolumn{1}{|l}{} & \multicolumn{1}{|l}{} & \multicolumn{1}{|l}{}
& \multicolumn{1}{|l}{$R$} & \multicolumn{1}{|l}{$D$} & \multicolumn{1}{|l|}{
} \\ \cline{2-7}\cline{4-4}
$y=3$ & \multicolumn{1}{|l}{$D$} & \multicolumn{1}{|l}{} & 
\multicolumn{1}{|l}{$R$} & \multicolumn{1}{|l}{} & \multicolumn{1}{|l}{} & 
\multicolumn{1}{|l|}{} \\ \cline{2-7}\cline{4-4}
$y=2$ & \multicolumn{1}{|l}{} & \multicolumn{1}{|l}{$R,D$} & 
\multicolumn{1}{|l}{} & \multicolumn{1}{|l}{} & \multicolumn{1}{|l}{} & 
\multicolumn{1}{|l|}{} \\ \cline{2-7}\cline{4-4}
$y=1$ & \multicolumn{1}{|l}{$R$} & \multicolumn{1}{|l}{} & 
\multicolumn{1}{|l}{$D$} & \multicolumn{1}{|l}{} & \multicolumn{1}{|l}{} & 
\multicolumn{1}{|l|}{} \\ \cline{2-7}\cline{4-4}
& $x=1$ & $x=2$ & $x=3$ & $x=4$ & $x=5$ & $x=6$%
\end{tabular}%
\]%
Divide $\mathbf{A}$ into four quadrants:%
\[
\begin{tabular}{ll}
lower left, LL & $\left\{ xy|1\leq x\leq 3,1\leq y\leq 3\right\} $ \\ 
lower right, LR & $\left\{ xy|4\leq x\leq 6,1\leq y\leq 3\right\} $ \\ 
upper left, UL & $\left\{ xy|1\leq x\leq 3,4\leq y\leq 6\right\} $ \\ 
upper right, UR & $\left\{ xy|4\leq x\leq 6,4\leq y\leq 6\right\} $%
\end{tabular}%
\]%
So $R$ and $D$ lie entirely in LL and UR

Thinking about how a line $L$ not parallel to a grid line $x=c$ or $y=d$
meets these grid lines, one sees that $L$ must have as many points in LL as
in UR. Moreover, $L$ cannot lie entirely in UL and LR , because it must meet
at least one of $R$ and $D$, since they cannot both be parallel to $L$. So
the five lines parallel to $R$ other than $R$ each have a point in LL. As
there are six points in LL not on $R$, four of these parallels have one
point in LL and one has two. The same thing holds for the parallels of $D$.

Now suppose that the points $45,56,64$ in UR not on $R$ or $D$ are
collinear, on $L$, say. Then the part of $L$ in LL would be one of these
four possibilities:%
\[
\begin{tabular}{llll}
$11$ & $23$ & $32$ & $R$ \\ 
$12$ & $21$ & $33$ & $R$ \\ 
$12$ & $23$ & $31$ & $D$ \\ 
$13$ & $21$ & $32$ & $D$%
\end{tabular}%
\]%
The fourth entry tells which of $R$ or $D$ the line $L$ meets. But each
choice presents a parallel to the other of $R$ or $D$, and that gives $L$
too many points in LL.

Thus one of the three sides of the triangle with vertices $45,56,64$ is not
parallel to either of $R$ or $D$. Such a side $S$ cannot meet either $R$ or $%
D$ in UR without being a grid line. So $S$ meets \emph{both} $R$ and $D$ in
LL. The only way to do that without being a grid line is to go through $22$;
but now there is no second point available for $S$ in LL that is not already
on a line through $22$.

Therefore the diagonals of every parallelogram in $\mathbf{A}$ are parallel.
Keep the grid layout above and the line $R$. Work with grid parallelograms
with vertices $xy,xy^{\prime },x^{\prime }y,x^{\prime }y^{\prime }$ to
determine points on lines, as follows: if we know one diagonal and know that
a parallel to it goes through a third vertex, then that parallel must go
through the fourth vertex. Line $R$ is a diagonal of the parallelogram $%
11,12,21,22$, and we now take $D$ to be the other diagonal, parallel to $R$.
If $D$ goes through $3z$, then by parallelogram $33,3z,z3,zz$ we find that $%
z3\in D$. Renumber to make $z=4$. Then $R$ and $D$ have these points, $56$
and $65$ now being forced to be on $D$:%
\begin{eqnarray*}
R &:&11,22,33,44,55,66 \\
D &:&12,21,34,43,56,65
\end{eqnarray*}%
Take $A$ and $B$ to be parallels to $R$ and $D$ with $13\in A$ and $14\in B$
initially. The further points of $A$ and $B$ shown in the diagram here are
obtained from the parallelogram rule.%
\[
\begin{tabular}{lllllll}
\cline{2-7}\cline{4-4}
$y=6$ & \multicolumn{1}{|l}{} & \multicolumn{1}{|l}{} & \multicolumn{1}{|l}{}
& \multicolumn{1}{|l}{} & \multicolumn{1}{|l}{$D$} & \multicolumn{1}{|l|}{$R$%
} \\ \cline{2-7}\cline{4-4}
$y=5$ & \multicolumn{1}{|l}{} & \multicolumn{1}{|l}{} & \multicolumn{1}{|l}{}
& \multicolumn{1}{|l}{} & \multicolumn{1}{|l}{$R$} & \multicolumn{1}{|l|}{$D$%
} \\ \cline{2-7}\cline{4-4}
$y=4$ & \multicolumn{1}{|l}{$B$} & \multicolumn{1}{|l}{$A$} & 
\multicolumn{1}{|l}{$D$} & \multicolumn{1}{|l}{$R$} & \multicolumn{1}{|l}{}
& \multicolumn{1}{|l|}{} \\ \cline{2-7}\cline{4-4}
$y=3$ & \multicolumn{1}{|l}{$A$} & \multicolumn{1}{|l}{$B$} & 
\multicolumn{1}{|l}{$R$} & \multicolumn{1}{|l}{$D$} & \multicolumn{1}{|l}{}
& \multicolumn{1}{|l|}{} \\ \cline{2-7}\cline{4-4}
$y=2$ & \multicolumn{1}{|l}{$D$} & \multicolumn{1}{|l}{$R$} & 
\multicolumn{1}{|l}{$B$} & \multicolumn{1}{|l}{$A$} & \multicolumn{1}{|l}{}
& \multicolumn{1}{|l|}{} \\ \cline{2-7}\cline{4-4}
$y=1$ & \multicolumn{1}{|l}{$R$} & \multicolumn{1}{|l}{$D$} & 
\multicolumn{1}{|l}{$A$} & \multicolumn{1}{|l}{$B$} & \multicolumn{1}{|l}{}
& \multicolumn{1}{|l|}{} \\ \cline{2-7}\cline{4-4}
& $x=1$ & $x=2$ & $x=3$ & $x=4$ & $x=5$ & $x=6$%
\end{tabular}%
\]%
For instance, parallelogram $11,13,31,33$ with diagonal $R$ and $13\in A$
implies that $31\in A$. Then $21,24,31,34$ with diagonal $D$ and $31\in A$
makes $24\in A$. Continuing this way, we fill in all the positions $xy$ with 
$1\leq x,y\leq 4$. But now we're stuck -- there's no place for two more $A$%
's and $B$'s. We conclude that $\mathbf{A}$ does not exist!

\section*{Comments}

The section of the \emph{Handbook of Combinatorial Designs} cited \cite[%
Section III.3]{HCD} presents general results on mutually orthogonal Latin
squares. There is a recent disproof of Euler's conjecture using certain
combinatorial matrices in the paper by K. Wang and K. Chen \cite{WC}.

Once it is known that the diagonals of parallelograms in affine planes of
order $6$ are parallel, it follows that the diagonal points of any
quadrangle in a projective plane of order 6 are collinear. Such a plane is a 
\textbf{Fano plane} in the terminology of A. Gleason \cite{G}. But then it
is Desarguesian, by Theorem 3.5 of \cite{G}. So its order would have to be a
power of 2.

The famous Bruck-Ryser theorem \cite{BR} rules out infinitely many orders
for affine planes, $6$ being one of them.


\begin{thebibliography}{99}
\bibitem{AK} E. F. Assmus, Jr., and J. D. Key, \emph{Designs and their codes}%
, Cambridge University Press, Cambridge (1962).

\bibitem{B} Alessandro Bichara,\emph{\ }An elementary proof of the
nonexistence of a projective plane of order six, \emph{Mitt. Math. Sem.
Giessen} No. 192 (1989) 89--93.

\bibitem{BR} R. H. Bruck and H. J. Ryser, The nonexistence of certain finite
projective planes, \emph{Canadian J. Math.} \textbf{\ 1} (1949) pp.. 88--93.

\bibitem{BPS} R. C. Bose, E. T. Parker, and S. S. Shrikhande, Further
results on the construction of mutually orthogonal Latin squares and the
falsity of Euler's conjecture, \emph{Canad. J. Math.} \textbf{12} (1960) pp.
189--203.

\bibitem{C} H. S. M. Coxeter, \emph{Regular polytopes (third edition)},
Dover Publications, Inc., New York (1973).

\bibitem{D} Steven T. Dougherty, A coding theoretic solution to the 36
officer problem, \emph{Des. Codes Cryptogr.} \textbf{4} (1994) 123--128.

\bibitem{E} L. Euler, Recherches sur une nouvelles esp\`{e}ce de quarr\'{e}s
magiques, \emph{Verhandelingen uitgegeven door het zeeuwsch Genootschap der
Wetenschappen te Vlissingen }\textbf{9} (1782), pp. 85--239 = \emph{Opera
Omnia}: Ser. 1, Vol. 7, pp. 291--392. Translation by Andie Ho and Dominic
Klyve:\newline
http://eulerarchive.maa.org/docs/translations/E530.pdf

\bibitem{G} Andrew M. Gleason, Finite Fano planes, \emph{Amer. J. Math}. 
\textbf{78} (1956) 797--807.

\bibitem{HCD} \emph{Handbook of Combinatorial Designs (second edition)},
Charles J. Colbourn and Jeffrey H. Dinitz, editors, Chapman \& Hall/CRC,
Boca Raton, FL (2007).

\bibitem{P} E. T. Parker, Orthogonal latin squares, \emph{Proc. Nat. Acad.
Sci. U.S.A}. \textbf{45} (1959) 859--862.

\bibitem{S} D. R. Stinson, A short proof of the nonexistence of a pair of
orthogonal Latin squares of order six, \emph{J. Combin. Theory } Ser. A%
\textbf{\ 36} (1984) 373--376.

\bibitem{T} G. Tarry, Le probl\`{e}me de 36 officiers, \emph{Compte Rendu de
l'Assoc. Fran\c{c}ais Avanc. Sci. Naturel} \textbf{2} (1901) 170-203.

\bibitem{WC} Kun Wang and Kejun Chen, A short proof of Euler's conjecture
based on quasi-difference matrices and difference matrices, \emph{Discrete
Math.} \textbf{341} (2018) 1114--1119.
\end{thebibliography}
\end{document}